\documentclass[a4paper,12pt]{amsart}
\usepackage[utf8]{inputenc}
\usepackage[T1]{fontenc}
\usepackage[UKenglish]{babel}
\usepackage[margin=18mm]{geometry}

\usepackage{color}

\usepackage{graphicx}

\usepackage{amsmath,amssymb,amsfonts,amsthm}
\usepackage{mathrsfs,eucal,dsfont}
\usepackage{verbatim,enumitem}

\usepackage{hyperref,url}

\newcommand{\R}{\mathds R}

\newcommand{\I}{\mathds 1}

\def\aa{\alpha}

\def\d{{\rm d}}
\def\<{\langle}
\def\>{\rangle}

 \def\ss{\sqrt}
\def\bb{\beta}

\def\R{\mathbb R}   \def\ss{\sqrt} 
 \def\kk{\kappa} 
  \def\vv{\varepsilon} 
\def\<{\langle} \def\>{\rangle}  \def\gg{\gamma}
  \def\nn{\nabla}  
\def\d{\text{\rm{d}}} \def\bb{\beta} \def\aa{\alpha} 
   
 \def\beq{\begin{equation}}  
 
\def\e{\text{\rm{e}}}    
  
 \def\P{\mathbb P}

  \def\ll{\lambda}
 
\def\i{{\rm in}}\def\E{\mathbb E} 
  
  \def\i{{\rm i}} 
\def\to{\rightarrow}
\def\8{\infty}\def\3{\triangle}
\def\1{\lesssim}

\renewcommand{\bar}{\overline}

\renewcommand{\tilde}{\widetilde}

\newtheorem{theorem}{Theorem}[section]
\newtheorem{lemma}[theorem]{Lemma}
\newtheorem{proposition}[theorem]{Proposition}

\theoremstyle{definition}

\newtheorem{example}[theorem]{Example}
\newtheorem{remark}[theorem]{Remark}

\numberwithin{equation}{section}
\begin{document}
\allowdisplaybreaks

\title[Ergodicity for damping Hamiltonian dynamics with non-local collisions] {Exponential ergodicity for  damping Hamiltonian dynamics with state-dependent and non-local collisions}

\author{
Jianhai Bao\qquad\,
\and\qquad
Jian Wang}
\date{}
\thanks{\emph{J.\ Bao:} Center for Applied Mathematics, Tianjin University, 300072  Tianjin, P.R. China. \url{jianhaibao@tju.edu.cn}}
\thanks{\emph{J.\ Wang:}
College of Mathematics and Statistic \&  Fujian Key Laboratory of Mathematical
Analysis and Applications (FJKLMAA) \&  Center for Applied Mathematics of Fujian Province (FJNU), Fujian Normal University, 350007 Fuzhou, P.R. China. \url{jianwang@fjnu.edu.cn}}

\maketitle

\begin{abstract}
 In this paper, we investigate  the exponential ergodicity in a Wasserstein-type distance for a damping Hamiltonian dynamics with state-dependent and non-local collisions, which indeed is a special case of piecewise deterministic Markov processes while
is very popular  in numerous modelling situations including stochastic algorithms. The approach adopted in this work is based on a combination of the refined basic coupling and the refined reflection coupling for non-local operators.
In a certain sense,  the main result developed in the present paper is a continuation of the  counterpart  in \cite{BW2022} on exponential ergodicity of stochastic Hamiltonian systems with L\'evy noises and a complement  of \cite{BA} upon exponential ergodicity for Andersen dynamics with constant jump rate functions.

\medskip


\smallskip

\end{abstract}

\section{Introduction}\label{section1}

\subsection{Background}
Piecewise deterministic Markov processes (PDMPs for short) constitute a very natural class of non-diffusive stochastic processes, where the mathematical framework was built by   Mark H. A. Davis in \cite{Davis}.
Roughly speaking, the PDMP is a process which jumps at some random time and moves continuously between  two adjacent random times; see \cite{Davis-book, Jac} for more details.
According to  \cite[Section 3]{Davis}, the probability law of a PDMP with the state space $E$ is determined by the following three ingredients: (i) a vector field $\Xi$, generating a deterministic flow; (ii) a jump rate function $J:E\to[0,\8)$, giving  the law of the random times between jumps; (iii)  a jump measure $Q:E\times E\to(0,\8)$ (i.e., for each fixed $A\in \mathscr B(E)$, $E\ni x\mapsto Q(x,A)$  is a measurable function, and,  for each fixed $x\in E$, $ \mathscr B(E)\ni A\mapsto Q(x,A)$  is a probability measure), giving the transition probability kernel of its jumps. The class of PDMPs is more general than compound Poisson processes and basic queues, and includes also jump processes over vector fields. PDMPs have a great variety of
applications such as in biology (cellular mass), physics (polymers length), computer science (TCP window size process), reliability (workload and repairable systems), mathematical finance, to name a few;
see, for instance, an excellent comprehensive survey paper \cite{Mal}
 on recent progresses of PDMPs and related open problems.
 Understanding the ergodic properties of these models from all areas above, in particular the distance under which (or the rate at which) they stabilize towards equilibrium, has in
turn increased the interest in the long-time behavior of PDMPs; see \cite{BBMZ, CoD, Cz, DGM} and references therein for the recent study.

In this paper, we consider a special class of PDMPs $(X_t,V_t)_{t\ge0}$ on the state space
$  \R^{2d}: =\R^d\times\R^d$ and associated with the following infinitesimal generator
\begin{equation}\label{E0}
\begin{split}
\big(\mathscr L f\big)(x,v)&=\big(\<\nn_x f(x,v),  v\>-\<\nn_vf(x,v), \gamma v+ \nn U(x)\>\big)\\
&\quad+J(x,v) \int_{\R^d}\big(f(x,u)-f(x,v)\big)\varphi(u)\,\d u\\
&=:(\mathscr L_{1,\gamma} f)(x,v)+(\mathscr L_2 f)(x,v),\quad\quad f\in C_b^{1}(\R^{2d}),
\end{split}
\end{equation}
where $\gamma>0,$ $U :\R^{ d}\to\R $ is smooth, $J:\R^{2d}\to(0,\8)$, and $\varphi(\cdot)$, which is
radial (i.e., $\varphi(x)=\varphi(|x|)$ for all $x\in \R^d$), is a probability density function on $\R^d$. In \eqref{E0},  $C_b^1(\R^{2d})$ means  the collection of  bounded real-valued functions $f(x,v)$ on $\R^{2d}$, which are   differentiable in $x$ and $v$, respectively, and $\nn_xf(x,v)$ and  $\nn_vf(x,v)$  denote  the first order gradients  of $f(x,v)$ with respect to the variable $x$ and the variable $v$, respectively.

Now, we make some detailed  expositions on the quantities involved in \eqref{E0}.
More precisely, $( v,-\gamma v-\nn U(x ))$ is the vector field generating the damping Hamiltonian flow,
 where $\gamma$  means the friction intensity that
 ensures a  damped-driven Hamiltonian  and
$-\gamma v$ stands for the damping force;
$J:\R^{2d}\to (0,\8)$ is the jump rate; $\varphi(u)\,\d u$ represents the jump measure. In terminology,
$\mathscr L_{1,\gamma} $ is called the Liouville operator associated with the damping Hamiltonian flow generated by the vector field $(x,-\gamma v-\nn U(x))$, and
$\mathscr L_2$ is the so-called
non-local collision operator. In particular, if $\varphi(u)$ is the density function of the standard normal distribution and $J(x,v)=\lambda$ for all $x,v\in \R^d$,   $\mathscr L_2$ is called the complete momentum randomization operator; see, for example, \cite{BS}.
It is worthy to emphasize that,
in statistical physics, the damping Hamiltonian system  has been applied widely  to model many vibration phenomena (e.g., the generalized Duffing oscillator); see e.g. \cite{Talay,Wu}. In the past two decades, great progresses upon long term behaviors  (e.g., ergodicity and large deviation) have been made for stochastic damping Hamiltonian systems; see, for instance, \cite{CLP,EGZ,LM,MSH} and references within for more details.

\subsection{Main result}
The purpose of this paper is to study the exponential ergodicity of the PDMP $(X_t,V_t)_{t\ge0}$ whose generator $\mathscr{L}$ is given by \eqref{E0}.  Before we  state our main result, we first present the assumptions. First of all, we assume that
\begin{enumerate}

\item[(${\bf H_0}$)]\it For any $\bb\in\R$, there exists a constant $K_{\bb,U}>0$ such that
for all $x,x'\in\R^d,$
\begin{equation*}
|\bb(x-x')+\nn U(x' )-\nn U(x )|\le K_{\bb,U} |x-x'|.
\end{equation*}

\end{enumerate}
\noindent In particular, $\nn U$ is Lipschitz continuous under ${\bf(H_0)}$.

For  the jump rate $J$ and the probability density  $\varphi$ of the jump measure, we assume  that
\begin{enumerate}
\item[(${\bf A_1}$)] {\it $J:\R^{2d}\to (0,\8)$ is uniformly bounded between two positive constants, i.e., there exist constants $\lambda_1,\lambda_2>0$ such that $\lambda_1\le J(x,v)\le \lambda_2$ for all $(x,v)\in\R^{2d}$.
Moreover, $J$ is globally Lipschitz continuous on $\R^{2d}$, i.e., there exists  a constant $\lambda_J>0$ such that for all $(x,v),(x',v')\in \R^{2d}$,
\begin{equation}\label{J1}
|J(x,v)-J(x',v')|\le \lambda_J\big(|x-x'|+|v-v'|\big).
\end{equation}}
\item[(${\bf A_2}$)] {\it For any $\alpha, \kappa>0$,
there exist $c_*(\alpha, \kappa), c^*(\alpha,\kk)>0$  such that for all $z\in \R^d$,
\begin{equation} \label{e12}
c_*( \alpha,\kappa)\le A_{\alpha,\kk}(z):=\int_{\R^d} \psi_{\alpha (z)_\kk }(u)\,\d u  \quad \mbox{ and } \quad 1- A_{\alpha,\kk}(z) \le   c^{*}(\alpha,\kk)|z|,
\end{equation}
where  for all $\xi,u\in\R^d$,
\begin{equation}\label{S1}
\psi_\xi(u):=\varphi(u)\wedge\varphi (u+\xi
),\end{equation}
and,
for the threshold $\kk>0$, the truncation counterpart of $z\in\R^d$ is defined by
\begin{equation}\label{S2}
 (z)_\kk=
  \frac{(\kk\wedge |z|)z}{|z|} \I_{\{z\neq{\bf0}\}}+{\bf0}\I_{\{z={\bf0}\}}.
\end{equation}}

\end{enumerate}
Since $A_{\alpha,\kappa}(0)=\int_{\R^d} \psi_0(u)\,\d u=\int_{\R^d}\varphi(u)\,\d u=1$, in some sense \eqref{e12} indicates the non-degenerate property and the continuity of the probability density  $\varphi$.

Besides all the assumptions above, we further need the following Lyapunov  condition:
\begin{itemize}
\item[(${\bf B_1}$)] {\it There exist a $C^1$-function $\mathcal W:\R^{2d}\to[1,\8)$
 and constants $c_0,C_0>0$ such that
\begin{equation}\label{EEE}
\lim_{|x|+|v|\to \infty} \mathcal W(x,v)=\infty
\end{equation} and  for all $(x,v)\in\R^{2d}$,
\begin{equation}\label{E1}
(\mathscr L \mathcal W)(x,v)\le-c_0 \mathcal W(x,v)+C_0.
\end{equation}}

\item[(${\bf B_2}$)] {\it There exists a constant $c^{**}>0$ such that for all $x,\xi\in\R^d,$
\begin{equation}\label{e11}
\int_{\R^d}\mathcal W(x,u)\varphi(u)\,\d u\le c^{**}\inf_{v\in\R^d}\mathcal W(x,v ),\quad \int_{\R^d}\mathcal W(x,u)\,\Psi_{\xi}(u)\, \d u \le c^{**}\inf_{v\in\R^d}\mathcal W(x,v ) |\xi|,
\end{equation}
where for all $\xi,u\in\R^d$,
\begin{equation}\label{S3}
\Psi_\xi(u):=\varphi(u)-\psi_\xi(u)
\end{equation}
with $\psi_\xi(u)$ being introduced in \eqref{S1}.}
\end{itemize}

Let $\mathscr P(\R^{2d})$ be the set of probability measures on $\R^{2d}$. For $\mu,\nu\in\mathscr P(\R^{2d})$, define the quasi-Wasserstein distance between $\mu$ and $\nu$ induced by a distance-like function $\Phi:\R^{2d}\times\R^{2d}\to[0,\8)$ (see e.g. \cite[Definition 4.3]{HMS}) as below
$$
  W_\Phi(\mu,\nu)=\inf_{\Pi\in \mathscr{C}(\mu,\nu)}\int_{\R^{2d}\times\R^{2d}} \Phi(x,y)\, \Pi(\d x,\d y),
$$
where $\mathscr{C}(\mu,\nu)$ stands for the collection of all couplings of $\mu$ and $\nu.$ In particular,  $W_\Phi$ goes back to the classical  Wasserstein distance when $\Phi$ is a metric function.
Note that $W_\Phi(\mu,\nu)=0$ if and only if $\mu=\nu $, since $\Phi$ is a distance-like function. Moreover, the space
\begin{equation*}
\mathscr P_\Phi(\R^{2d}):=\Big\{\mu\in \mathscr P(\R^{2d}):\int_{\R^{2d}}\Phi(x,{\bf 0})\,\mu(\d x)<\8\Big\}
\end{equation*}
is complete under  $W_\Phi$, i.e., each $W_\Phi$-Cauchy sequence in $\mathscr P_\Phi(\R^{2d})$ 
converges with respect to $W_\Phi$.

For each $t\ge0$, let $P_t\big((x,v),\cdot\big)$ be the transition probability kernel of the Markov process $(X_t,V_t)_{t\ge0}$ with initial value $(X_0,V_0)=(x,v)$ associated with the generator $\mathscr L. $   Furthermore, we shall write $\mu P_t$ to mean the distribution of $(X_t,V_t) $ with   initial distribution $\mu\in\mathscr P(\R^{2d})$.

 The main result of this paper is stated as follows.

\begin{theorem}\label{main}
Assume that $( {\bf H_0})$, $( {\bf A_1})$, $( {\bf A_2})$, $( {\bf B_1})$ and $( {\bf B_2})$ hold, and that  the following inequality
  \begin{equation}\label{S6}
  \bb\ge 4K_{\bb,U}
 \end{equation}
 is solvable in the interval $(0,\gamma^2/4]$, where $\gamma$ was given in \eqref{E0} and $K_{\bb,U}$ was given in $({\bf H_0})$.
  Then,  the PDMP $(X_t,V_t)_{t\ge0}$ corresponding to the operator $\mathscr L$ in \eqref{E0} is exponentially ergodic in the sense that there exist a unique invariant probability measure $\mu\in\mathscr P_\Phi(\R^{2d})$
and a constant $\lambda^*>0$ such that for any $\nu\in\mathscr P_\Phi(\R^{2d})$ and $t\ge0,$
\begin{equation}\label{e:main-1}
W_{\Phi}\big(\nu P_t,\mu\big)\le C(\mu,\nu)\e^{-\lambda^*t},
\end{equation}
where for all $(x,v),(x',v')\in\R^{2d}$,
\begin{equation}\label{e:main-2}
\Phi\big((x,v),(x',v')\big):=\big((|x-x'|+|v-v'|)\wedge 1\big)\big(\mathcal W(x,v)+\mathcal W(x',v')\big)
\end{equation} and $C(\mu,\nu)$ is a positive   function
depending on $\mu$ and $\nu$ $($indepedent of $t)$.
\end{theorem}

To illustrate the effectiveness of Theorem \ref{main}, we consider the following example.
\begin{example}\label{Example}Assume that Assumption ${\bf (A_1)}$ holds. Let $U(x)=\theta |x|^2$ with
$$ \frac{\gamma^2}{8}\ge \theta>
\frac{(\lambda_1+\gamma)^2(\lambda_2-\lambda_1)^2}{4(2\lambda_1\lambda_2-\lambda_1^2
+4\lambda_2\gamma+3\gamma^2)},$$ and $\varphi(x)=\varphi_1(x):=c_{d,\beta_1} (1+|x|)^{-d-\beta_1}$ with $\beta_1>0$ or $\varphi(x)=\varphi_2(x):= c_{d,\beta_2} \exp (-|x|^{\beta_2})$ with $\beta_2>0$. Then, the conclusion of Theorem \ref{main} holds with $\mathcal W(x,v)=(1+|x|^2+|v|^2)$ 
and
the previously defined $\varphi_2$ or $\varphi_1$ when $\beta_1>2$,
 and with $\mathcal W(x,v)=(1+|x|^2+|v|^2)^{(\beta_1-\varepsilon)/2}$ for any $\varepsilon\in(0,\bb_1)$
and the foregoing $\varphi_1$ when $\beta_1\in (0,2]$.
 \end{example}
\subsection{Comments} Recently, plenty  of interests have
grown concerning  the application of PDMPs to sample from a target distribution (for example, algorithms are referred to as PDMP
Monte Carlo (PDMP-MC) methods). Therefore, much more efforts are devoted to
the ergodicity and the other long time behaviour of the PDMPs; see e.g. \cite{AD,BG,DGM,FBK,LT}     and references therein. Our work is related to the existing  result \cite{BS}
and the recent one \cite{BA}. In \cite{BS}, the exponential ergodicity for a randomized Hamiltonian Monte
Carlo (also called Hybrid Monte Carlo) was treated  under the same conditions that imply geometric ergodicity of the solution to underdamped Langevin equations.
The proof of \cite{BS} is based on a Foster–Lyapunov drift condition, a minorization condition and Harris' theorem. Via a coupling approach, the convergence to
equilibrium of Andersen dynamics (which becomes exact randomized Hamiltonian Monte Carlo when the associated molecular system consists of only one particle) was handled in \cite{BA}.
 As in \cite{BA,BA-add},
 we herein
also adopt  the probabilistic coupling method, whereas  the setting is significantly different from those in \cite{BA,BS}. For example,
\begin{itemize}
\item[(i)] The jump rate function in \cite{BA,BA-add,BS} is a constant function and moreover the jump measure is the standard normal distribution, though the non-local collisions involved in the PDMPs in \cite{BA,BA-add,BS} are much more general.
     Moreover, the
  exponential ergodicity in a Wasserstein sense of Andersen dynamics  was addressed in \cite{BA} nevertheless the position component was confined in a high-dimensional torus.  According to the private communications with Nawaf Bou-Rabee, the issue on  ergodicity of Andersen dynamics, where not only the velocity component but also the position component are supported on the whole Euclidean space, is highly non-trivial. Furthermore, we would like to emphasize that the exponential ergodicity of Andersen dynamics surviving  on the whole Euclidean space was investigated in \cite{BA-add},
 where the semi-metric inducing the Wasserstein-type distance admits the following form: for all $(x,v),(x',v')\in\R^{2d}$,
  \begin{equation*}
  \Phi\big((x,v),(x',v')\big)=\big((|x-x'|+|v-v'|)\wedge 1\big)\big(|x-x'|^2+|v-v'|^2\big)
  \end{equation*}
while the counterpart designed  in Theorem \ref{main} is a multiplicative type distance-like function (see \eqref{e:main-2} for more details) so the quasi-metric  involved in \cite{BA-add} is essentially different from the one  we exploited in Theorem \ref{main}.
 \item[(ii)] In the present paper, most importantly, we focus on  the state-dependent jump rate function. Additionally, we can not only deal with (sub-)Gaussian probability measures but also  the probability measures with heavy tails such like $\varphi(u)\,\d u=c_{d,\beta} (1+|u|)^{-d-\beta}\,\d u$ for $\beta>0$. Due to the appearance of the state-dependent jump rate function, compared with \cite{BA,BA-add,BS}, some additional sacrifices need to be paid.  Throughout the paper,
    the price to pay is that we will require that the constant $\gamma$ in the operator \eqref{E0} is positive; that is, we merely work on the damping Hamiltonian flow in our paper.
    On the other hand, the main results in \cite{BA,BA-add} require the constant jump rate function $J(x,v)$ is large enough while here in Theorem \ref{main} we do not need such kind condition even for the setting of non-constant jump rate functions.
    Therefore, from these points of view above, the results of \cite{BS,BA} and our paper  complement  each other.
\end{itemize}

The approach of our paper is also motivated partly by our previous work \cite{BW2022} on  exponential ergodicity of stochastic Hamiltonian systems with L\'evy noises.
However, in contrast to  \cite{BW2022} the non-local collision operator in the present setting is not only highly degenerate but also   state-dependent so much more delicate work  are to be implemented.  In particular, we shall adopt a combination of the refined basic coupling and the refined reflection coupling (rather than the refined basic coupling exploited merely in \cite{BW2022}) in order to include  more general probability measures (e.g.,
(sup-)Gaussian or (sub-)Gaussian probability
measures and probability measures with heavy tails). So, in a certain sense, Theorem \ref{main} is a continuation of the corresponding main result in \cite{BW2022} on exponential ergodicity of stochastic Hamiltonian systems with L\'evy noises. Furthermore, we emphasize that
the process under investigation in this paper
has some essentially different properties from stochastic Hamiltonian systems with L\'evy noises under consideration in \cite{BW2022}. For example, under some regular conditions the process associated with stochastic Hamiltonian systems with L\'evy noises can possess the strong Feller property; see \cite{Zhang}. Nonetheless, since the non-local collision operator $\mathscr L_2$ in \eqref{E0} is a bounded operator on $B_b(\R^{2d})$ under Assumption ${\bf (A_1)}$,  the PDMP $(X_t,V_t)_{t\ge0}$ corresponding to the operator $\mathscr L$ in \eqref{E0} can never enjoy the strong Feller property.

\ \

The rest of the paper is arranged as follows. In the next section, we construct a coupling operator and examine  the existence of the associated coupling process. Section \ref{Section3} is devoted to the proof of Theorem \ref{main}. In the last section, we present some sufficient conditions to guarantee that   Assumptions and the technical condition \eqref{S6} involved in  Theorem \ref{main} are verifiable.

\section{Coupling operator and coupling process}\label{Section2}
We start with some notations.
Let $I_{d\times d}$ be the $d\times d$ identity matrix, and $\R^d\otimes\R^d$ be the set of all $d\times d$ matrices. For $x\in\R^d,$
we write $x\otimes x =xx^*\in\R^d\otimes\R^d$ with  $x^*$ being its transpose.
For   $x\in\R^d,$ define the following
orthogonal matrix
\begin{equation}\label{EE1}
\Pi_{x}=
\Big(I_{d\times d}-2\Big(\frac{x}{|x|}\otimes\frac{x}{|x|}\Big)\Big)\I_{\{x\neq{\bf0}\}}-I_{d\times d} \I_{\{x={\bf0}\}}\in\R^d\otimes\R^d.
\end{equation}
For any $a,b\in\R$, let $a^+=\max\{a,0\}$, i.e., the positive part of the number $a$, and $a\wedge b=\min\{a,b\}$.

Fix $\alpha, \kappa>0$. For ${\bf y}=\big((x,v),(x',v')\big)\in\R^{4d} $ and $f\in C_b^{1}(\R^{4d})$, define the following operator
\begin{equation}\label{E2}
\big(\tilde{\mathscr L}_{\gamma,\alpha, \kappa}f\big) ({\bf y})=\big(\tilde{\mathscr L}_{1,\gamma}f\big)({\bf y})+\big(\tilde{\mathscr L}_{2,\alpha,\kappa}f\big)({\bf y}),
\end{equation}
where
\begin{equation}\label{W4}
\begin{split}
\big(\tilde{\mathscr L}_{1,\gamma}f\big)({\bf y}):&=\big\<\nn_x f({\bf y}), v\big\>+\big\<\nn_{x'} f({\bf y}), v'\big\>\\
&\quad-\big\<\nn_v f({\bf y}),  \gamma v+\nn U(x)\big\>-\big\<\nn_{v'} f({\bf y}),  \gamma v'+\nn U(x')\big\>
\end{split}
\end{equation}
and
\begin{equation}\label{W55}
\begin{split}
\big(\tilde{\mathscr L}_{2,\alpha, \kappa}f\big)({\bf y})
&:= \big(J(x,v)\wedge J(x',v')\big)\\
&\qquad\times\bigg\{\int_{\R^d}\big(f\big((x,u),(x',u+\alpha(x-x')_\kk
)\big)-f({\bf y})\big)\psi_{\alpha(x-x')_\kk }(u)\,  \d u\\
&\quad+   \int_{\R^d}\big(f\big((x,u),(x',\Pi_{ (x-x')_\kk }u
)\big)-f({\bf y})\big)\Psi_{\alpha(x-x')_\kk }(u)\, \d u\bigg\}\\
&\quad+\big(J(x,v)  -  J(x',v')\big)^+ \int_{\R^d}\big(f\big((x,u),(x',v'
)\big)-f({\bf y})\big)\psi(u)\,  \d u\\
&\quad  + \big(J(x',v')  -  J(x,v)\big)^+\int_{\R^d}\big(f\big((x,v),(x',u
)\big)-f({\bf y})\big) \psi(u)\,  \d u,
\end{split}
\end{equation}
where, for $z\in\R^d$, $(z)_\kk$ was defined as in \eqref{S1}, and $\psi_\xi(\cdot)$ and $\Psi_\xi(\cdot)$
were introduced in \eqref{S2} and \eqref{S3}, respectively.
 It is easy to see that the last two items on the right hand side of \eqref{W55} vanish once the jump rate $J$ is a constant function (i.e., $J(x,v)=\lambda$ for all $(x,v)\in\R^{2d}$ and some $\lambda>0$).

\begin{remark}  As shown in Lemma \ref{lem0} below, for any $\gamma,\alpha,\kappa>0$, $\tilde{\mathscr L}_{\gamma,\alpha,\kappa}$ is a coupling operator of $\mathscr L$. Indeed, for the operator $ \mathscr L_{1,\gamma} $, we adopt the synchronous coupling as showed in \eqref{W4}.
The coupling operator $\tilde{\mathscr L}_{2 ,\alpha,\kappa}$ associated with $ \mathscr L_2 $ is indeed built based on a combination of the refined basic coupling and the refined reflection coupling as well as the independent coupling:
\begin{equation*}
\big((x,v),(x',v')\big)\rightarrow
\begin{cases}
\big((x,u),(x',u+\alpha(x-x')_\kk
)\big),\quad&\big(J(x,v)\wedge J(x',v')\big) \psi_{\alpha(x-x')_\kk}(u)\,  \d u\\
\big((x,u),(x',\Pi_{ (x-x')_\kk }u)\big),  & \big(J(x,v)\wedge J(x',v')\big)\Psi_{\alpha(x-x')_\kk}(u)\,  \d u\\
\big((x,u),(x',v')\big),&\big(J(x,v)  -  J(x',v')\big)^+\varphi(u)\,\d u\\
\big((x,v),(x',u)\big), & \big(J(x',v')  -  J(x,v)\big)^+\varphi(u)\,\d u.
\end{cases}
\end{equation*} See e.g. \cite{LSW} or \cite{LW} for more details.
In particular, concerning  the coupling counterpart of $\tilde{\mathscr L}_{2,\alpha,\kappa}$, the velocity components change accordingly while
the position components remain unchanged. More precisely, the velocity component $(v,v')$ changes into $(u,u+\alpha(x-x')_\kk)$ with the maximum common intensity measure $ \big(J(x,v)\wedge J(x',v')\big) \psi_{\alpha(x-x')_\kk}(u)\,  \d u$;  the velocity component $(v,v')$ moves to the   point $(u,\Pi_{ (x-x')_\kk }u)$ with the  intensity measure $\big(J(x,v)\wedge J(x',v')\big)\Psi_{\alpha(x-x')_\kk}(u)\,  \d u$; the velocity component $(v,v')$ changes  to $(u,v')$ and $(v,u)$ with the remainder intensity measures $\big(J(x,v)  -  J(x',v')\big)^+\varphi(u)\,\d u$ and $\big(J(x',v')  -  J(x,v)\big)^+\varphi(u)\,\d u$, respectively, to guarantee the marginal property of the coupling operator $\tilde{\mathscr L}_{2 ,\alpha,\kappa}$ defined by \eqref{W55}.  Moreover, it is worthy to
 stress that the construction above heavily depends on the
 radial property of $\varphi$.
\end{remark}

\begin{lemma}\label{lem0}
For any $\gamma,\alpha, \kappa>0$, the operator $\tilde{\mathscr L}_{\gamma,\alpha, \kappa}$, defined in \eqref{E2}, is a coupling operator of $\mathscr L$, introduced in \eqref{E0}.
\end{lemma}

\begin{proof} For simplicity, we shall write $\tilde{\mathscr L}_{\gamma,\alpha, \kappa}$, $ \tilde{\mathscr L}_{1,\gamma}$, and   $ \tilde{\mathscr L}_{2,\alpha,\kappa}$ as $\tilde{\mathscr L}$,  $ \tilde{\mathscr L}_{1}$ and $\tilde{\mathscr L}_2$, respectively.
To demonstrate  that $\tilde{\mathscr L}$ is a coupling operator, we  only need to verify that  $\tilde{\mathscr L}_1$ and  $\tilde{\mathscr L}_2$
are coupling operators corresponding to $ \mathscr L_{1,\gamma}$ and  $ \mathscr L_2$, respectively. To achieve this, it is sufficient to prove that
for any $f\in C_b^1(\R^{4d})$ so that $f({\bf y})=g(x,v)+h(x',v')$ with some $h,g\in C_b^1(\R^{2d})$ and for any ${\bf y}=\big((x,v),(x',v')\big) $,
\begin{equation}\label{E3}
\big(\tilde{\mathscr L}_1f\big) ({\bf y})=(\mathscr L_{1,\gamma} g)(x,v)+ (\mathscr L_{1,\gamma} h)(x',v'),
\end{equation}
and
\begin{equation}\label{E4}
 \big(\tilde{\mathscr L}_2f\big) ({\bf y})=(\mathscr L_2 g)(x,v)+ (\mathscr L_2 h)(x',v').
\end{equation}
It is trivial to see that \eqref{E3} holds true. On the other hand, according to the definition  of  $\tilde{\mathscr L_2}$, we deduce  that
\begin{align*}
\big(\tilde{\mathscr L}_2f\big) ({\bf y})&=\big(J(x,v)\wedge J(x',v')\big)
\int_{\R^d}\big(g(x,u)-g(x,v)\big)\psi_{\alpha(x-x')_\kk }(u)\,  \d u\\
&\quad + \big(J(x,v)\wedge J(x',v')\big)  \int_{\R^d}\big(g(x,u) -g(x,v)\big)\Psi_{\alpha(x-x')_\kk }(u)\, \d u \\
&\quad+\big(J(x,v)  - J(x',v')\big)^+ \int_{\R^d}\big(g(x,u) -g(x,v)\big)\psi(u)\,  \d u\\
&\quad+\big(J(x,v)\wedge J(x',v')\big) \int_{\R^d}\big(h(x',u+\alpha(x-x')_\kk
) -h(x',v') \big)\psi_{\alpha(x-x')_\kk }(u)\,  \d u \\
&\quad+
\big(J(x,v)\wedge J(x',v')\big)\int_{\R^d}\big(h(x',\Pi_{ (x-x')_\kk }u
) -h(x',v')\big)\Psi_{\alpha(x-x')_\kk }(u)\, \d u\\
&\quad + \big(J(x',v')  - J(x,v)\big)^+  \int_{\R^d}\big(h(x',u
)-h(x',v')\big)\psi(u)\,  \d u \\
&=(\mathscr L_2 g)(x,v)\\
&\quad+\big(J(x,v)\wedge J(x',v')\big)\int_{\R^d}\big(h(x',u
) -h(x',v') \big)\psi_{-\alpha(x-x')_\kk }(u)\,  \d u\\
&\quad+ \big(J(x,v)\wedge J(x',v')\big)\int_{\R^d}\big(h(x',\Pi_{ (x-x')_\kk }u
) -h(x',v')\big)\Psi_{\alpha(x-x')_\kk }(u)\, \d u\\
&\quad+\big(J(x',v')  -J(x,v) \big)^+\int_{\R^d}\big(h(x',u
)-h(x',v')\big)\psi(u)\,  \d u,
\end{align*}
where in the second identity we took advantage of the definition of $\Psi_\cdot$, and used the basic identity: $a\wedge b+(a-b)^+=a$ for any $a,b\in \R$, as well as
substituted the variable $ u+\alpha(x-x')_\kk $ with the variable $u$.
 Note that the matrix $\Pi_\cdot$, defined in \eqref{EE1}, is an orthogonal matrix and its inverse $\Pi_{\cdot}^{-1}$ is equal to $\Pi_{\cdot}$.   Thus, we find
\begin{align*}
u+\alpha(x-x')_\kk &=\Pi_{(x-x')_\kk}^{-1}\big(\Pi_{ (x-x')_\kk }u+\alpha\Pi_{ (x-x')_\kk }(x-x')_\kk \big)\\
&=\Pi_{ (x-x')_\kk } \big(\Pi_{ (x-x')_\kk }u-\alpha (x-x')_\kk \big).
\end{align*}
This, along with the
radial property of $\varphi$ and $\Pi_{\cdot}^{-1}=\Pi_{\cdot}$, gives us that  for  any  mapping $\Theta:\R^{d} \to\R, $
\begin{equation}\label{EE3}
\begin{split}
 \int_{\R^d} \Theta\big( \Pi_{ (x-x')_\kk}u
\big) \Psi_{\alpha (x-x')_\kk}(u)\, \d u
 &= \int_{\R^d} \Theta\big(  \Pi_{  (x-x')_\kk}u
\big)  \Psi_{-\alpha (x-x')_\kk }(\Pi_{  (x-x')_\kk}u)\,\d u\\
&= \int_{\R^d} \Theta(  u
) \Psi_{-\alpha (x-x')_\kk }(u)\, \d u.
\end{split}
\end{equation}
The identity above enables us to obtain
\begin{align*}
&\big(J(x,v)\wedge J(x',v')\big)\int_{\R^d}\big(h(x',\Pi_{ (x-x')_\kk }u
) -h(x',v')\big)\Psi_{\alpha(x-x')_\kk }(u)\, \d u\\
&= \big(J(x,v)\wedge J(x',v')\big)\int_{\R^d}\big(h(x', u
) -h(x',v')\big)\Psi_{-\alpha(x-x')_\kk }(u)\, \d u.
\end{align*}
Consequently, we have
\begin{align*}
\big(\tilde{\mathscr L}_2f\big) ({\bf y})&=(\mathscr L_2 g)(x,v) \\
&\quad +\big(J(x,v)\wedge J(x',v')\big)  \int_{\R^d}\big(h(x',u
) -h(x',v') \big)\psi_{-\alpha(x-x')_\kk }(u)\,  \d u\\
&\quad + \big(J(x,v)\wedge J(x',v')\big)  \int_{\R^d}\big(h(x', u
) -h(x',v')\big)\Psi_{-\alpha(x-x')_\kk }(u)\, \d u \\
&\quad+\big(J(x',v')  -J(x,v) \big)^+\int_{\R^d}\big(h(x',u
)-h(x',v')\big)\psi(u)\,  \d u\\
&=(\mathscr L_2 g)(x,v)+ (\mathscr L_2 h)(x',v'),
\end{align*}
where in the second identity we used again the definition of $\Psi_\cdot$ and
the fact: $a\wedge b+(a-b)^+=a$ for any $a,b\in \R$.
Therefore,
\eqref{E4} is now available.
\end{proof}

Before we end this section, we address the issue on the existences of a Markovian coupling process associated with the coupling operator $\tilde{\mathscr L}_{\gamma,\alpha, \kappa}$.
To achieve this goal, we set on $\R^{4d}$
a vector field $\bar{\Xi}:=( v, v',-\gamma v-\nn U(x ), -\gamma v'-\nn U(x' ))$ and a jump measure
\begin{align*}\mathcal{Q}(x,x',v,v',\d y, \d y', \d u, \d u')= & \big(J(x,v)\wedge J(x',v')\big)\delta_{\{y=x,y'=x',u'=u+\alpha(x-x')_\kk\}} \psi_{\alpha(x-x')_\kk}(u)\,  \d u\\
&+ \big(J(x,v)\wedge J(x',v')\big)\delta_{\big\{y=x,y'=x',u'=\Pi_{ (x-x')_\kk }u\big\}}\Psi_{\alpha(x-x')_\kk}(u)\,  \d u\\
&+\big(J(x,v)  -  J(x',v')\big)^+\delta_{\{y=x,y'=x',u'=v'\}}\varphi(u)\,  \d u\\
&+ \big(J(x',v')  -  J(x,v)\big)^+\delta_{\{y=x,y'=x',u=v\}}\varphi(u)\,  \d u.\end{align*}
Under ${\bf(A_1)}$, it is clear that,  for any $(x,x',v,v')\in \R^{4d}$, $\mathcal{Q}(x,x',v,v',\d y, \d y', \d u, \d u')$ is a finite measure $\R^{4d}$. Furthermore, we set
a jump rate function $\mathcal{J}(x,x',v,v')=  \mathcal{Q}(x,x',v,v',\R^{4d})$ and define a normalized jump measure $$\bar{\mathcal{Q}}(x,x',v,v',\d y, \d y', \d u, \d u')=\mathcal {J}(x,x',v,v')^{-1} \mathcal{Q}(x,x',v,v',\d y, \d y', \d u, \d u').$$
Subsequently, according to  \cite[Section 3]{Davis}, there exists an $\R^{4d}$-valued PDMP $\big((X_t,V_t),(X_t',V_t')\big)_{t\ge0}$ corresponding to the triplet $(\bar{\Xi},\mathcal{J}(x,x',v,v'), \bar{\mathcal{Q}}(x,x',v,v',\d y, \d y', \d u, \d u'))$. Obviously, the generator of $\big((X_t,V_t),(X_t',V_t')\big)_{t\ge0}$ above is nothing else but  the coupling operator $\tilde{\mathscr L}_{\gamma,\alpha, \kappa}$.
This   proves the existence of a Markovian coupling process associated with the coupling operator $\tilde{\mathscr L}_{\gamma,\alpha, \kappa}$ and  we therefore  reach our desired goal.

\section{Proof of Theorem \ref{main}}\label{Section3}
Throughout this section, we shall write ${\bf y}=\big((x,v),(x',v')\big)$ for all $(x,v),(x',v')\in\R^{2d}$.
For  the parameters $\alpha, \aa_0>0$ (whose precise values are to be given later), we    introduce the following  abbreviated notations:
\begin{equation}\label{3D}
z:=x-x',\quad w:=v-v',\quad q  := z +\alpha^{-1}w,\quad r({\bf y}) :=\alpha_0 |z|  +|q |.
\end{equation}
For any $\vv>0$ and $(x,v),(x',v')\in\R^{2d}$, set
\begin{equation}\label{E5}
F ({\bf y}): =f(r({\bf y})),\quad\quad G( {\bf y}): =1+\vv\,  \mathcal W(x,v)+\vv\,\mathcal W(x',v'),
\end{equation}
where the $C^2$-function $f:[0,\8)\to [0,\8)$  satisfies  $f(0) = 0$, $f'\ge0$ and $ f''\le0$, and $\mathcal W$ is  the Lyapunov function  given in ($ {\bf B_1}  )$.

\begin{lemma}\label{lem1}
For $F,G$, given in \eqref{E5},
\begin{equation}\label{E8}
\big(\tilde{\mathscr L}_{\gamma,\alpha,\kappa}(FG)\big) ({\bf y})=G({\bf y})\big(\tilde{\mathscr L}_{1,\gamma}F\big)({\bf y})+F({\bf y})\big(\tilde{\mathscr L}_{\gamma,\alpha,\kappa}G\big)({\bf y})+\Pi({\bf y}),
\end{equation}
where
\begin{equation}\label{F1}
\begin{split}
\Pi({\bf y}):&= \big(J(x,v)\wedge J(x',v')\big)\\
    &\quad\quad\times\int_{\R^d}\big( F\big((x,u),(x',u+\alpha (z)_\kk )\big)- F ({\bf y})\big)  G\big((x,u),(x',u+\alpha  (z)_\kk  ) \big)\,\psi_{\alpha (z)_\kk  }(u)\,   \d u \\
  &\quad+  \big(J(x,v)\wedge J(x',v')\big)\\
    &\quad\quad\times  \int_{\R^d}    \big( F\big((x,u),(x',\Pi_{ (z)_\kk  }u)\big)-F({\bf y})\big) G\big((x,u),(x',\Pi_{ (z)_\kk  }u)\big) \Psi_{\alpha  (z)_\kk  }(u)\,   \d u \\
 &\quad+  \big(J(x,v)  -  J(x',v')\big)^+ \int_{\R^d}  \big( F\big((x,u),(x',v'
)\big)- F({\bf y})\big) G\big((x,u),(x',v'
)\big)\,\psi(u)\,   \d u\\
 &\quad+  \big(J(x',v')  -  J(x,v)\big)^+ \int_{\R^d} \big( F\big((x,v),(x',u
)\big)- F({\bf y})\big)  G\big((x,v),(x',u
)\big) \,\psi(u)\,   \d u.
\end{split}
\end{equation}
\end{lemma}

\begin{proof}
Apparently, the chain rule yields that for all ${\bf y}\in\R^{4d}$,
\begin{equation}\label{E6}
\big(\tilde{\mathscr L}_{1,\gamma}(FG)\big)({\bf y})  = F({\bf y})\big(\tilde{\mathscr L}_{1,\gamma}G\big)({\bf y})+G({\bf y})\big(\tilde{\mathscr L}_{1,\gamma}F\big)({\bf y}).
\end{equation}
Next, by invoking the addition-subtraction strategy and taking the definition of the operator $\tilde{\mathscr L_2}_{,\alpha,\kappa}$ into account,   we derive that  for all ${\bf y}\in\R^{4d}$,
\begin{equation}\label{E7}
\begin{split}
\big(\tilde{\mathscr L}_{2,\alpha,\kappa}(FG)\big)({\bf y})
& =\Pi({\bf y})\\
&\quad +
\big(J(x,v)\wedge J(x',v')\big) F({\bf y})\\
&\qquad\times\bigg\{\int_{\R^d}\big( G\big((x,u),(x',u+\alpha  (z)_\kk
)\big)- G_\vv({\bf y})\big)\,\psi_{\alpha  (z)_\kk  }(u)\,   \d u
\\
&\qquad\qquad+ \int_{\R^d}  \big( G\big((x,u),(x',\Pi_{ (z)_\kk  }u)\big)- G ({\bf y})\big)\Psi_{\alpha (z)_\kk  }(u)\,   \d u\bigg\}\\
&\quad+ \big(J(x,v)  -  J(x',v')\big)^+ F ({\bf y})  \int_{\R^d}\big( G\big((x,u),(x',v'
)\big)- G({\bf y})\big)\,\psi(u)\,   \d u\\
&\quad+ \big(J(x',v')  -  J(x,v)\big)^+ F ({\bf y})\int_{\R^d}\big( G\big((x,v),(x',u
)\big)- G({\bf y})\big)\,\psi(u)\,   \d u \\
&=F({\bf y})\big(\tilde{\mathscr L}_{2,\alpha,\kappa}G\big)({\bf y})+\Pi({\bf y}),
\end{split}
\end{equation}
where the remainder term  $\Pi(\cdot)$ was introduced in \eqref{F1}.
Thus, recalling $ \tilde{\mathscr L}_{\gamma,\alpha,\kappa} = \tilde{\mathscr L}_{1,\gamma} + \tilde{\mathscr L}_{2,\alpha,\kappa} $ and combining \eqref{E6} with \eqref{E7} enables us to derive \eqref{E8}.
\end{proof}

\medskip

In the following, we assume that Assumption   $({\bf B_1})$ holds. Let $\mathcal W(x,v)$ and $c_0, C_0$ be the function and the constants in  $({\bf B_1})$.
Define the following two sets
\begin{equation}\label{6D}
\mathcal A =\big\{{\bf y}\in\R^{4d}: 4\,  C_0 \ge c_0 \mathcal W(x,v)+c_0\mathcal W(x',v')\big\}, \quad \Gamma =\big\{{\bf y}\in\R^{4d}: r({\bf y})\ge R_0\big\},
\end{equation}
where
\begin{equation}\label{e6}
R_0=R_0(\alpha,\alpha_0): =\sup\big\{r({\bf y}): {\bf y}\in \mathcal A   \big\}.
\end{equation}
 Due to \eqref{EEE} (i.e., $\lim_{|x|+|v|\to \infty} \mathcal W(x,v)=\infty$),
 there is an  $R^*>0$ (independent of $\alpha_0,\alpha$ but dependent on $c_0$ and $C_0$) such that
$|x|+|v|\le R^*$ for all $(x,v)\in \mathcal A_0$, where
$$\mathcal A_0:=\big\{(x,v)\in\R^{2d}: \mathcal W(x,v)\le  4C_0/c_0\big\}.$$
It is trivial to see that  $\mathcal A$ is a subset of the product space $\mathcal A_0\times \mathcal A_0.$
Hence, we find that
\begin{align*}
R_0:=\sup_{{\bf y}\in \mathcal A}r({\bf y}) &\le\big(  1+\alpha_0 + \alpha^{-1} \big)\sup_{{\bf y}\in \mathcal A} \big(|x|+|v|+|x'| +|v'|\big)\\
&  \le2\big( 1+ \alpha_0 + \alpha^{-1} \big)\sup_{(x,v)\in \mathcal A_0}  \big(|x|+|v|\big )\\
&\le2R^*\big(  1+\alpha_0 + \alpha^{-1} \big).
\end{align*}
As a consequence, $R_0$ can be bounded by the number $R^*\big(  1+ \alpha_0 + \alpha^{-1} \big)$ up to an absolute constant independent of $\alpha_0,\alpha$.
On the other hand, it follows from the definitions of $\mathcal A$ and $\Gamma$ that $\mathcal A \subset \Gamma^c$.

Now, we set for all ${\bf y}\in\R^{4d}$,
\begin{equation*}
\tilde{F}(r({\bf y})):=f(r({\bf y})\wedge R_0), \qquad
\end{equation*}
where $f(\cdot)$ was given in \eqref{E5}.

\begin{lemma}\label{lem2}
Under Assumption   $({\bf B_1})$,  for  all
${\bf y}\in \mathcal A^c \cap \Gamma=\Gamma$,
\begin{equation} \label{e8}
\big(\tilde{\mathscr L}_{\gamma,\alpha,\kappa}(\tilde{F}G)\big) ({\bf y})\le -\frac{c_0 \vv}{1+2\vv} \tilde{F}({\bf y})G({\bf y}).
\end{equation}
\end{lemma}

\begin{proof}
For ${\bf y}\in \mathcal A^c \cap \Gamma$ (in particular, $r({\bf y})\ge R_0$), $f(r({\bf y})\wedge R_0)=f(R_0)$ and $f'_-(r({\bf y})\wedge R_0)=0$, where $f'_-$ means the left derivative of $f$.
The chain rule shows that
$$
\big(\tilde{\mathscr L}_{1,\gamma}\tilde{F}\big)({\bf y})=f'_-(r({\bf y})\wedge R_0)\,\big(\tilde{\mathscr L}_{1,\gamma}r\big)({\bf y})=0
$$
 and so
\begin{equation}\label{e4}
G({\bf y})\big(\tilde{\mathscr L}_{1,\gamma}\tilde{F}\big)({\bf y})=0.
\end{equation}
Next,  in addition to $f'>0$ on $[0,\8)$ and the positive  properties of  $J$ and $\mathcal W$, we find that $\Pi({\bf y})\le0$ once $r({\bf y})\ge R_0$.

On the other hand, by applying Lemma \ref{lem0} and taking  \eqref{E1} into consideration, we have
\begin{equation}\label{e5}
\begin{split}
\tilde{F}({\bf y})\big(\tilde{\mathscr L}_{\gamma,\alpha,\kappa}G\big)({\bf y})&=\vv f(r({\bf y})\wedge R_0)\big((\mathscr L \mathcal W)(x,v)+(\mathscr L \mathcal W)(x',v')\big)\\
&\le \vv f(r({\bf y})\wedge R_0)\big(-c_0 \mathcal W(x,v)-c_0 \mathcal W(x',v')+2C_0   \big).
\end{split}
\end{equation}
Thus, combining \eqref{e4} with \eqref{e5} and making use  of  Lemma \ref{lem1} leads  to
\begin{align*}
\big(\tilde{\mathscr L}_{\gamma,\alpha,\kappa}(\tilde{F}G)\big) ({\bf y})&\le   \vv f( R_0)\big(-c_0 \mathcal W(x,v)-c_0 \mathcal W(x',v')+2C_0    \big)
\end{align*}
in case of $r({\bf y})\ge R_0.$ Subsequently, for all ${\bf y}\in \mathcal A^c,$ we obviously have
$$\frac{c_0 }{2} \big(\mathcal W(x,v)+\mathcal W(x',v')\big)\ge 2C_0.  $$
Accordingly, we arrive at
\begin{equation}\label{e7}
\big(\tilde{\mathscr L}_{\gamma,\alpha,\kappa}(\tilde{F}G)\big) ({\bf y})\le  -\frac{c_0\vv}{2} f( R_0)   \big( \mathcal W(x,v)+\mathcal W(x',v')\big) .
\end{equation}
Additionally, due to $\mathcal W\ge1$, we evidently have
\begin{equation}\label{e700}
G ({\bf y})=1+\vv \big(\mathcal W(x,v)+\mathcal W(x',v')\big)\le (1/2+\vv)\big(\mathcal W(x,v)+\mathcal W(x',v')\big).
\end{equation}
As a result, concerning  the case $r({\bf y})\ge R_0$, we obtain from \eqref{e7} that
$$
\big(\tilde{\mathscr L}_{\gamma,\alpha,\kappa}(\tilde{F}G)\big) ({\bf y})\le  -\frac{c_0 \vv}{1+2\vv}  f( R_0)  G ({\bf y})=-\frac{c_0 \vv}{1+2\vv}  \tilde{F}({\bf y})G({\bf y}).
$$
Hence, the desired assertion \eqref{e8} follows.
\end{proof}

Now, for any $a_0>0$ (which will be fixed later), we take the function $f$    in \eqref{E5} to be
\begin{equation}\label{E17}
f(s)=\frac{1}{a_0}\big(1-\e^{-a_0s}\big),\quad s\ge0,
\end{equation}
which definitely satisfies that $f(0)=0$, $f'>0$ and $f''<0$ on $[0,\8).$
Moreover, simple calculations  yield  the following two crucial estimates:
\begin{equation}\label{E18}
f(s)-f(t)\le f'(t)(s-t),\quad\quad s,t\ge0
\end{equation}
and
\begin{equation}\label{E19}
f(s)-f(t)\le \frac{1}{a_0}f'(t),\quad\quad s,t\ge0;
\end{equation}
see, for instance, \cite[Lemma 5.2]{BA} for more details.

In the remainder of the paper, we shall fix the threshold $\kk$ in the coupling operator $\tilde{\mathscr L}_{\gamma,\alpha,\kappa}$, defined in \eqref{E2},  as below $$\kappa=R_0/\alpha_0.$$

 \begin{lemma}\label{lem4}
Under   Assumptions $( {\bf A_1})$, $( {\bf A_2})$, $($${\bf B_1}$$)$ and $($${\bf B_2}$$)$,  for all ${\bf y}\in \Gamma^c,$
\begin{equation}\label{W5}
\begin{split}
 \big(\tilde{\mathscr L}_{\gamma,\alpha,\kappa}(FG)\big) ({\bf y})
&\le G   ({\bf y}) f'_-(r({\bf y}) )\,\big(\tilde{\mathscr L}_{1,\gamma}r\big)({\bf y})-\lambda_1 c_*(\alpha,\kappa) f'_{-}(r({\bf y}))    |q|\\
&\quad + \left[\frac{1}{a_0}\big(\lambda_2\big(c^*(\alpha,\kappa)\vee(c^{**} \aa)\big)+2\lambda_J(1+\alpha) (1\vee c^{**})\big) \right]f'_{-}(r({\bf y}))G({\bf y})|z|\\
&\quad+ \frac{2\alpha }{a_0}\lambda_J(1\vee c^{**})f'_{-}(r({\bf y}))G({\bf y})|q|\\
&\quad+\vv f(r({\bf y}) )\big(- c_0\big(\mathcal W(x,v)+\mathcal W(x',v')\big)
+2C_0    \big).
 \end{split}
\end{equation}
\end{lemma}

\begin{proof}
For ${\bf y}\in \Gamma^c$ (i.e., $r({\bf y})<R_0$), the chain rule yields
\begin{equation} \label{YY5}
\big(\tilde{\mathscr L}_{1,\gamma}\tilde{F}\big)({\bf y})=f'_-(r({\bf y}))\,\big(\tilde{\mathscr L}_{1,\gamma}r\big)({\bf y}).
\end{equation}
 By virtue of \eqref{e5},  we readily have
\begin{equation}\label{EE2}
\tilde{F}({\bf y})\big(\tilde{\mathscr L}_{\gamma,\alpha,\kappa} G\big)({\bf y})
\le \vv f(r({\bf y}) )\big(-c_0    \mathcal W(x,v)-c_0  \mathcal W(x',v')+2C_0   \big) .
\end{equation}

Write down the four terms on the right hand side of \eqref{F1} as $\Upsilon_1({\bf y}), \Upsilon_2({\bf y}), \Upsilon_3({\bf y})$ and $\Upsilon_4({\bf y})$, respectively. Below, we intend to quantify $\Upsilon_i({\bf y}), i=1,2,3,4, $  separately.
According to the definition of $r(\cdot)$, we deduce
\begin{align*}
\Upsilon_1({\bf y})&= \big(J(x,v)\wedge J(x',v')\big)  \big(f( \alpha_0+ (1-(1\wedge\kk/|z|)) |z|) -f(r({\bf y}))\big)\\
  &\quad\times \int_{\R^d}    G\big((x,u),(x',u+\alpha  (z)_\kk ) \big)\,\psi_{\alpha (z)_\kk }(u)\,  (\d u)\\
& \le  \big(J(x,v)\wedge J(x',v')\big) f'_{-}(r({\bf y}))\\
&\quad\times\int_{\R^d}    G\big((x,u),(x',u+\alpha  (z)_\kk ) \big)\,\psi_{\alpha (z)_\kk }(u)\,  (\d u)( \alpha_0  |z|-r({\bf y}))\\
&=   - |q| \big(J(x,v)\wedge J(x',v')\big) f'_{-}(r({\bf y}))\int_{\R^d}    G\big((x,u),(x',u+\alpha  (z)_\kk)  \big)\,\psi_{\alpha (z)_\kk }(u)\,  (\d u) \\
&\le  - \lambda_1 c_*(\alpha,\kappa) f'_{-}(r({\bf y})) |q|,
\end{align*}
where in the first inequality we used the fact $\kappa=R_0/\alpha_0> r({\bf y})/\alpha_0\ge|z|$, in the identity we utilized $r({\bf y})=\alpha_0|z|+|q|$, and
the last inequality is available due to  (${\bf B_2}$) and $G\ge1$.

Next,
by invoking \eqref{E19},
together with  $f'>0$ on $[0,\8)$, we obtain    that
\begin{align*}
 \Upsilon_2({\bf y})
  & \le  \frac{1}{a_0} \big(J(x,v)\wedge J(x',v')\big)f'_{-}(r({\bf y}))\int_{\R^d}G\big((x,u),(x',\Pi_{ (z)_\kk  }u)\big)\,\Psi_{\alpha (z)_\kk}(u)\, \d u\\
  &\le  \frac{\lambda_2}{a_0} f'_{-}(r({\bf y})) \bigg\{1-A_{\alpha,\kk}(z)+\vv\int_{\R^d}\mathcal W(x,u)\,\Psi_{\alpha (z)_\kk}(u)\, \d u\\
  &\qquad\qquad\qquad\quad+\vv\int_{\R^d}\mathcal W(x',\Pi_{ (z)_\kk  }u)\,\Psi_{\alpha (z)_\kk}(u)\, \d u\bigg\}\\
   &= \frac{\lambda_2}{a_0} f'_{-}(r({\bf y})) \bigg\{1-A_{\alpha,\kk}(z)+\vv\int_{\R^d}\mathcal W(x,u)\,\Psi_{\alpha (z)_\kk}(u)\, \d u
   \\
& \qquad\qquad\qquad \quad +\vv\int_{\R^d}\mathcal W(x',u)\,\Psi_{-\alpha (z)_\kk}(u)\, \d u\bigg\}\\
  &\le \frac{\lambda_2}{a_0} f'_{-}(r({\bf y}))\Big \{1-A_{\alpha,\kk}(z)+\vv c^{**} \aa(\kk\wedge |z|)\Big(\inf_{v\in\R^d}\mathcal W(x,v)+\inf_{v'\in\R^d}\mathcal W(x',v')\Big)\Big\}\\
  &\le \frac{\lambda_2}{a_0}\big(c^*(\alpha,\kk)\vee(c^{**} \aa)\big) f'_{-}(r({\bf y}))G({\bf y})|z|,
\end{align*}
where the  identity is due to \eqref{EE3}, the last two inequality holds true  owing to  (${\bf B_2}$), and the last display follows from (${\bf A_2}$). Once more, using  \eqref{E19} yields
 \begin{align*}
 \Upsilon_3({\bf y})+ \Upsilon_4({\bf y})
&\le\frac{f'_{-}(r({\bf y}))}{a_0}\big(J(x,v)  -  J(x',v')\big)^+ \Big(1+\vv\mathcal W(x',v')+\vv\int_{\R^d}\mathcal W(x,u)\varphi(u)\,   \d u \Big)\\
&\quad+\frac{f'_{-}(r({\bf y}))}{a_0}\big(J(x',v')  -  J(x,v)\big)^+ \Big(1+\vv\mathcal W(x,v)+\vv\int_{\R^d}  \mathcal W(x',u)   \varphi(u)\,   \d u\Big)\\
&\le  \frac{\lambda_J f'_{-}(r({\bf y}))}{a_0}\big((1+\alpha)|z|+\alpha|q|\big)\\
&\quad\times \Big(2+\vv \mathcal W(x,v)+\vv \mathcal W(x',v')+c^{**}\vv\Big(\inf_{v\in\R^d}\mathcal W(x,v)+ \inf_{v'\in\R^d}\mathcal W(x',v')\Big)\Big)\\
&\le \frac{2\lambda_J}{a_0}(1\vee c^{**})\big((1+\alpha)|z|+\alpha|q|\big) f'_{-}(r({\bf y}))G({\bf y}),
 \end{align*}
 where in the second inequality we exploited \eqref{J1} and (${\bf B_2}$) as well as $w=\alpha(q-z)$.

 Consequently, we complete the proof of Lemma \ref{lem4}  by
 combining all the estimates above for $\Upsilon_i({\bf y})$ $(1\le i\le 4)$ with \eqref{YY5} and \eqref{EE2}.
\end{proof}

\ \

From now on, we assume that the inequality \eqref{S6} is solvable on the interval $(0,\gamma^2/4]$. Then,  there exists  $\bb\in(0,\gamma^2/4]$ solving the inequality
$
\bb\ge 4K_{\bb,U}.
$
Due to $\bb\in(0,\gamma^2/4]$, there exists an $\alpha>0$ such that $\bb=\alpha\gamma-\alpha^2.$
Hence, the   inequality
$$
\alpha\gamma-\alpha^2\ge 4 K_{\alpha(\gamma-\alpha),U}
$$
is also solvable. That is,
\begin{equation}\label{S4}
{\aa}^{-1}{\gamma}-1\ge 4\alpha^{-2}K_{\alpha(\gamma-\alpha),U}.
\end{equation}

In the sequel, we settle out the parameters  involved in \eqref{W55}, \eqref{3D} and \eqref{E5},  respectively.
 More precisely, for the positive constant $\aa$, a solution  to \eqref{S4}, we shall stipulate
\begin{equation}\label{1E}
\begin{split}
  \alpha_0 &= \frac{\gg}{\aa}-1,\quad \kk=\frac{R_0}{\alpha_0},\quad
a_0=\frac{4K_0}{\alpha_0\alpha}+4 \bigg(\frac{1}{\lambda_1c_*(\aa,\kk)}\vee\frac{2}{c_0}\bigg)\alpha\lambda_J (1\vee c^{**}),\\
\varepsilon&=\Bigg(\frac{c_0}{4C_0}\bigg(\frac{\lambda_1c_*(\aa,\kk)a_0}{4\alpha\lambda_J (1\vee c^{**})}-1\bigg)\Bigg)\wedge\bigg(\frac{a_0R_0}{8C_0(\e^{a_0R_0}-1)}\min\{\aa, \lambda_1  c_*(\aa,\kappa)  \}\bigg),
\end{split}
\end{equation}
where $R_0$ was defined by \eqref{e6}, and
$$
K_0:=\lambda_2\big(c^*(\alpha,\kappa)\vee(c^{**} \aa)\big)+2\lambda_J(1+\alpha) (1\vee c^{**}).$$
According to the prescribed value of $a_0$, it is evident to see that the value of $\vv$ set in \eqref{1E} is positive. Seemingly, the parameters set in \eqref{1E} are a little bit weird and complicated while the precise  alternatives will become more and more clear by tracking the proof of Lemma \ref{lem3} below.

  \begin{lemma}\label{lem3}
Assume that  $( {\bf H_0})$, $( {\bf A_1})$, $( {\bf A_2})$, $( {\bf B_1})$ and $( {\bf B_2})$ hold, and  that the inequality \eqref{S6} is solvable.
Then for all ${\bf y}\in
\Gamma^c,$
\begin{equation}\label{D4}
\big(\tilde{\mathscr L}_{\gamma,\alpha,\kappa}(\tilde{F}G)\big) ({\bf y})\le -\lambda^*
  \tilde{F}({\bf y})G({\bf y}),
\end{equation}
where $$\lambda^*:=
\frac{c_0\vv}{2(1+2\vv)}\wedge\bigg( \frac{a_0R_0}{4(\e^{a_0R_0}-1)}\min\{\aa, \lambda_1  c_*(\aa,\kappa)\} \left(1+4\varepsilon C_0/c_0\right)^{-1}\bigg).$$
\end{lemma}

\begin{proof}
A direct   calculation yields
\begin{align*}
\big(\tilde{\mathscr L}_{1,\gamma}r\big)({\bf y})&= \frac{\alpha_0}{|z|}\<z, w\>+ \frac{1}{|q|}\big\<q, -( \alpha^{-1}\gamma-1)w+\alpha^{-1}(  \nn U(x')- \nn U(x))\big\>  \\
  &=- \alpha_0\alpha |z|+\frac{ \alpha_0\alpha }{|z|}\<z,q\> -(\gamma-\alpha)|q|+\frac{1}{\alpha|q|}\big\<q,  \alpha( \gamma-\alpha)z+   \nn U(x')- \nn U(x) \big\>\\
  &\le \big( -\alpha_0 \alpha+  \alpha^{-1}K_{\alpha(\gamma-\alpha),U})|z|
+( \alpha+\alpha_0\alpha-\gamma)|q|\\
&=\big( -\alpha_0 \alpha+  \alpha^{-1}K_{\alpha(\gamma-\alpha),U}\big)|z|,
\end{align*}
where in the second identity  we utilized  the identity $w=\alpha(q-z)$,  in the inequality we employed $({\bf H_0})$, and in the last identity we took the fact that $\alpha+\alpha_0\alpha-\gamma=0$ due to \eqref{1E} into account.
Plugging the previous inequality back into \eqref{W5} implies   that for all ${\bf y}\in \Gamma^c$,
\begin{equation}\label{D2}
\begin{split}
\big(\tilde{\mathscr L}_{\gamma,\alpha, \kappa}(\tilde{F}G)\big) ({\bf y})&\le \Big(-  \aa_0 \alpha
+ \alpha^{-1}K_{\alpha(\gamma-\alpha),U}+\frac{K_0}{a_0} \Big)G   ({\bf y}) f'_-(r({\bf y}) )|z|\\
&\quad- \lambda_1  c_*(\alpha,\kappa)  f'_{-}(r({\bf y})) |q|+\frac{2\alpha}{a_0}\lambda_J (1\vee c^{**})G   ({\bf y})f'_{-}(r({\bf y}))|q|\\
&\quad+ \vv f(r({\bf y}) )\big(- c_0 \mathcal W(x,v)-c_0\mathcal W(x',v')
+2C_0   \big).
 \end{split}
\end{equation}
In terms  of \eqref{S4} and  \eqref{1E},  we obviously have
 \begin{equation*}
\alpha^{-1}K_{\alpha(\gamma-\alpha),U}\le \frac{\aa}{4}\Big(\frac{\gamma}{\alpha}-1\Big)=\frac{ \alpha_0\aa }{4}
, \quad \frac{K_0}{a_0}\le  \frac{  \alpha_0\aa }{4}. \end{equation*}
Then, \eqref{D2} is reduced into
\begin{equation}\label{D2-}
\begin{split}
\big(\tilde{\mathscr L}_{\gamma,\alpha, \kappa}(\tilde{F}G)\big) ({\bf y})
&\le  -\frac{ 1}{2}\aa_0 \aa G   ({\bf y}) f'_-(r({\bf y}) )|z|- \lambda_1  c_*(\aa,\kappa)  f'_{-}(r({\bf y})) |q|\\
&\quad+ \frac{2\alpha}{a_0}\lambda_J (1\vee c^{**})G   ({\bf y})f'_{-}(r({\bf y}))|q|\\
&\quad+ \vv f(r({\bf y}) )\big(- c_0 \mathcal W(x,v)-c_0\mathcal W(x',v')
+2C_0   \big).
 \end{split}
\end{equation}
In what follows, we aim to show that \eqref{D4} is verifiable  for two separate cases.

\smallskip

\noindent{\bf Case (i):} ${\bf y}\in \mathcal A \cap\Gamma^c$ (i.e., ${\bf y}\in \mathcal A $). For such case, we  in particular have
$G({\bf y})\le 1+4\varepsilon C_0/c_0$. In the light of the precise value of $\vv$ given in \eqref{1E}, we obtain  that
\begin{equation}\label{e:eps}\begin{split} \frac{2\alpha}{a_0}\lambda_J (1\vee c^{**})\left(1+4\varepsilon C_0/c_0\right)\le \frac{1}{2}\lambda_1c_*(\alpha,\kk).\end{split}\end{equation}
Thus, from \eqref{D2-} and $G\ge1$, we find that for all $y\in {\mathcal A}$,
\begin{align*}
\big(\tilde{\mathscr L}_{\gamma,\alpha, \kappa}(\tilde{F}G)\big) ({\bf y})&\le  -\frac{ 1}{2}\aa_0 \aa G   ({\bf y}) f'_-(r({\bf y}) )|z|- \lambda_1  c_*(\aa,\kappa)  f'_{-}(r({\bf y})) |q|\\
&\quad+ \frac{2\alpha}{a_0}\lambda_J (1\vee c^{**})\left(1+4\varepsilon C_0/c_0\right)f'_{-}(r({\bf y}))|q|+2C_0\vv f(r({\bf y}) )\\
&\le  -\frac{1}{2}  \aa_0\aa f'_-(r({\bf y}) )|z| -\frac{1}{2} \lambda_1  c_*(\aa,\kappa) f'_{-}(r({\bf y}))|q| + 2C_0 \vv f(r({\bf y}) )\\
&\le -\frac{1}{2}\min\{\aa, \lambda_1  c_*(\aa,\kappa)  \}f'_{-}(r({\bf y})) r({\bf y}) + 2C_0\vv f(r({\bf y}) ),
 \end{align*}
where in the last display we used the fact that $r({\bf y})=\alpha_0|z|+|q|$. Subsequently,
combining  the following facts that
\begin{equation}\label{S5}
 f'(s)s=\frac{a_0s}{\e^{a_0s}-1}f(s)\le f(s) ,\quad s\ge0
\end{equation}
and that $s\mapsto \frac{a_0s}{\e^{a_0s}-1}$ is decreasing on $[0,\8)$, leads to
\begin{equation}\label{D11}
\begin{split}
\big(\tilde{\mathscr L}_{\gamma,\alpha, \kappa}(\tilde{F}G)\big) ({\bf y})\le&  -\frac{1}{2}\min\{\alpha, \lambda_1  c_*(\alpha,\kappa)\}f'_{-}(r({\bf y})) r({\bf y})+  \frac{2C_0\vv (\e^{a_0R_0}-1)}{a_0R_0} f'_{-}(r({\bf y})) r({\bf y})\\
\le&  -\frac{1}{4}\min\{\alpha, \lambda_1  c_*(\alpha,\kappa)  \} f'_{-}(r({\bf y})) r({\bf y})\\
\le&  - \frac{a_0R_0}{4(\e^{a_0R_0}-1)}\min\{\alpha, \lambda_1  c_*(\alpha,\kappa)  \}f(r({\bf y}))\\
\le& - \frac{a_0R_0}{4(\e^{a_0R_0}-1)}\min\{\alpha, \lambda_1  c_*(\alpha,\kappa)  \} \left(1+4\varepsilon C_0/c_0\right)^{-1}G({\bf y})f(r({\bf y})),
 \end{split}
\end{equation}
where in the first inequality we invoked the basic fact $r({\bf y})\le R_0$ due to ${\bf y}\in \mathcal A,$
in the second inequality we employed the fact
 $$ \frac{2 \varepsilon C_0(\e^{a_0R_0}-1)}{a_0R_0}\le \frac{1}{4}\min\{\aa, \lambda_1  c_*(\aa,\kappa)\} $$
by taking the alternative  of $\vv$, given in \eqref{1E}, into consideration, and
in the last inequality we applied  \eqref{e700}.

\noindent{\bf Case (ii):} ${\bf y}\in \mathcal A^c \cap\Gamma^c$. Concerning this setting,  we have  $r({\bf y})<R_0$ and
\begin{equation}\label{S7}
   c_0 \mathcal W(x,v)+c_0\mathcal W(x',v')\ge 4 C_0.
\end{equation}
From \eqref{e:eps}, it is obvious to see that
\begin{equation*}
\frac{2\alpha}{a_0}\lambda_J (1\vee c^{**}) \le \frac{1}{2}\lambda_1c_*(\alpha,\kk).
\end{equation*}
Thus,
we derive  from \eqref{D2-} and  \eqref{S7} that for all ${\bf y}\in \mathcal A^c \cap\Gamma^c$,
\begin{equation} \label{7E}
\begin{split}
\big(\tilde{\mathscr L}_{\gamma,\alpha, \kappa}(\tilde{F}G)\big) ({\bf y})&\le  -\frac{1}{2} \aa_0 \aa G   ({\bf y}) f'_-(r({\bf y}) )|z|- \lambda_1  c_*(\aa,\kappa)  f'_{-}(r({\bf y})) |q|\\
&\quad+ \frac{2\alpha}{a_0}\lambda_J (1\vee c^{**}) f'_{-}(r({\bf y}))|q|\\
&\quad + \frac{2\alpha}{a_0}\lambda_J (1\vee c^{**})\vv\big(   \mathcal W(x,v)+  \mathcal W(x',v')  \big)f'_{-}(r({\bf y}))|q|\\
&\quad -\frac{c_0\vv}{2} \big(   \mathcal W(x,v)+ \mathcal W(x',v') \big) f(r({\bf y}) )\\
&\le-\frac{1}{2}\min\{\aa, \lambda_1  c_*(\aa,\kappa) \}f'_{-}(r({\bf y})) r({\bf y}) \\
&\quad+ \frac{c_0 \vv}{4}\big(   \mathcal W(x,v)+  \mathcal W(x',v')  \big)f'_{-}(r({\bf y}))|q|\\
&\quad-\frac{c_0 \vv}{2} \big(  \mathcal W(x,v)+  \mathcal W(x',v') \big) f(r({\bf y}) )\\
&\le-   \frac{c_0\vv}{4} \big(   \mathcal W(x,v)+  \mathcal W(x',v')  \big)f(r({\bf y}) )\\
&\le-\frac{ c_0\vv}{2(1+2\vv)} G({\bf y})f(r({\bf y}) ),
 \end{split}
\end{equation} where
in the first inequality we utilized the fact that $G({\bf y})=1+\vv\big(\mathcal W(x,v)+\mathcal W(x',v')\big)$,
in the second inequality we exploited  $G\ge1$, $r({\bf y})=\alpha_0|z|+|q|$ and
\begin{equation*}
 \frac{2\alpha}{a_0}\lambda_J (1\vee c^{**})\le \frac{c_0}{4}
\end{equation*}
with the help of \eqref{1E},
 in the third  inequality we used the basic fact that $q\le r({\bf y})$ and \eqref{S5}, and in the last inequality we took advantage of
 \eqref{e700} again.

Consequently, the assertion \eqref{D4} follows immediately  by combining \eqref{D11} with \eqref{7E}.
\end{proof}

Next, combining Lemma \ref{lem2} with Lemma \ref{lem3}, we arrive at the following proposition.

\begin{proposition}\label{pro1}
Assume that  $( {\bf H_0})$, $( {\bf A_1})$, $( {\bf A_2})$, $( {\bf B_1})$ and $( {\bf B_2})$ hold, and that the inequality \eqref{S6} is solvable. Then, for all ${\bf y}\in\R^{4d}$,
\begin{equation*}
\big(\tilde{\mathscr L}_{\gamma,\alpha,\kappa}(\tilde{F}G)\big) ({\bf y})\le -\lambda^*
  \tilde{F}({\bf y})G({\bf y}),
\end{equation*}
where $$\lambda^*:=
\frac{c_0\vv}{2(1+2\vv)}\wedge\bigg( \frac{a_0R_0}{4(\e^{a_0R_0}-1)}\min\{\aa, \lambda_1  c_*(\aa,\kappa)\} \left(1+4\varepsilon C_0/c_0\right)^{-1}\bigg).$$
\end{proposition}

\ \

Now, we are in position to complete the proof of Theorem \ref{main}. Although, with  Proposition \ref{pro1} at hand,
the proof of Theorem \ref{main} is more or less standard,
 we herein provide an outline to make the content self-contained.
\begin{proof}[Proof of Theorem $\ref{main}$]

Let $({\bf Y}_t)_{t>0}=\big((X_t,V_t),(X_t',V_t')\big)_{t\ge0}$ be the coupling process associated with the coupling operator $\tilde{\mathscr L}_{\gamma,\alpha,\kappa}$ as mentioned at the end of Section \ref{Section2}, and let $\tilde{\E}^{{\bf y}}$  be the expectation operator under the probability measure
$\tilde{\P}^{{\bf y}}$,  the distribution of $({\bf Y}_t)_{t\ge0}$ with the initial point ${\bf y}$. Then, we deduce from Proposition \ref{pro1} that
\begin{equation}\label{F3}
\tilde{\E}^{\bf y}(\tilde{F}G)({\bf Y}_t)\le (\tilde{F}G)({\bf Y}_0)\e^{-\lambda^*t}.
\end{equation}
Note that $(\tilde{F}G)({\bf y}) $ is comparable with the quasi-distance function
$$\Phi({\bf y}):=\big((|x-x'|+|v-v'|)\wedge 1\big)\big(\mathcal W(x,v)+\mathcal W(x',v')\big);$$
that is, there exist constants $c_1,c_2>0$  such that for all ${\bf y}\in \R^{4d}$,
$$c_1\Phi ({\bf y})\le (\tilde{F}G)({\bf y}) \le c_2\Phi({\bf y}).$$
This, together with \eqref{F3}, yields that there is a constant $c_3>0$ such that for all ${\bf y}\in \R^{4d}$ and $t>0$,
\begin{equation}\label{F4}
W_{\Phi}\big(\delta_{(x,v)}P_t,\delta_{(x',v')}P_t \big)\le c_3\Phi({\bf y})\e^{-\lambda^* t},
\end{equation}
which further implies that the semigroup $(P_t)_{t\ge0}$ exhibits the Feller property, and moreover, via the Banach fixed point theorem, has a unique invariant probability measure $\mu\in \mathscr P_\Phi(\R^{2d})$; see, for instance,  \cite[Corollary 4.11]{HMS} for more details.  Now, for any $\nu\in \mathscr P_\Phi(\R^{2d})$, integrating with respect to $\pi\in\mathscr C(\nu,\mu)$ on both sides of \eqref{F4} yields
\begin{equation*}
\int_{\R^{2d}\times\R^{2d}}W_{\Phi}\big(\delta_{(x,v)}P_t,\delta_{(x',v')}P_t \big)\pi(\d {\bf y})\le c_3\e^{-\lambda^* t}\int_{\R^{2d}\times\R^{2d}}\Phi({\bf y})\pi(\d {\bf y}),
\end{equation*}
which,  combining  the basic fact that
\begin{equation*}
W_{\Phi}\big(\nu P_t,\mu P_t\big)\le\int_{\R^{2d}\times\R^{2d}}W_{\Phi}\big(\delta_{(x,v)}P_t,\delta_{(x',v')}P_t \big)\pi(\d {\bf y}) ,
\end{equation*}
and taking infimum with respect to all  $\pi\in\mathscr C(\nu,\mu)$ leads to
\begin{equation*}
W_{\Phi}\big(\nu P_t,\mu P_t\big)\le c_3\e^{-\lambda^* t} W_{\Phi}\big(\nu  ,\mu \big).
\end{equation*}
Thus, \eqref{e:main-1} follows by   taking the invariance of the invariant probability measure  $\mu$ into consideration. We therefore complete the corresponding proof.
\end{proof}

\section{Sufficient conditions on Assumptions and \eqref{S6} }
In this section, we aim to provide some sufficient conditions or concrete examples to demonstrate that all the assumptions and the technical condition \eqref{S6} are verifiable.

\begin{proposition}\label{pro} Suppose that the function $r\mapsto \varphi(r)$ is bounded and decreasing on $(0,\infty)$ such that $\varphi(r)>0$ for all $r>0$. Then, Assumption ${\bf(A_2)}$ holds.
 \end{proposition}
\begin{proof}
Since $\varphi(\cdot)$ is decreasing on $(0,\8)$, we deduce that
for any $z\in \R^d$,
\begin{align*}\int_{\R^d} \psi_z(u)\,\d u\ge& \int_{\R^d} \big(\varphi(|u|)\wedge \varphi(|u|+|z|)\big)\,\d u
\ge \int_{\R^d} \varphi(|u|+|z|)\,\d u=\int_{\{|u|\ge |z|\}}\varphi(u)\,\d u,\end{align*} which implies that for $r>0$,
$$\inf_{z\in \R^d: |z|\le r} \int_{\R^d}\psi_z(u)\,\d u\ge  \int_{\{|u|\ge r\}}\varphi(u)\,\d u>0$$
and that for all $z\in \R^d$, $$1-\int_{\R^d} \psi_z(u)\,\d u\le \int_{\{|u|\le  |z|\}}\varphi(u)\,\d u=c_0\int_0^{|z|}r^{d-1}\varphi(r)\,\d r\le c_1|z|$$
with some constants $c_0,c_1>0.$
Therefore, Assumption ${\bf(A_2)}$ holds.\end{proof}

The Lyapunov function $\mathcal W$ satisfying \eqref{EEE} and \eqref{E1} is imposed to analyze the
long-time behavior of the process $(X_t,V_t)_{t\ge0}$. Below, we build  examples to demonstrate that \eqref{EEE} and \eqref{E1} are valid. Suppose that $U(x)\ge0$ for all $x\in \R^d$. Let
\begin{equation}\label{2D}
\mathcal W_0(x,v)=1+2 U(x)+\theta_0|x|^2+ |v|^2+\theta^*\<x,v\>,\quad x,v\in\R^d,
\end{equation}
where
\begin{equation}\label{S9}
\theta_0:=\frac{1}{4} (\lambda_1+\gamma)^2,
\quad \theta^*:=\frac{(\lambda_1+\gamma)^2}{2(\lambda_2+\gamma)}
\end{equation} with $\lambda_1$ and $\lambda_2$ being given in Assumption ${\bf(A_1)}$.
Due to $\lambda_1\le \lambda_2,$ it is easy to see from \eqref{S9} that
$$
(\theta^*)^2=\frac{ (\lambda_1+\gamma)^4}{4(\lambda_2+\gamma)^2}\le \theta_0.
$$
By the  inequality that $2ab\le \vv a^2+b^2/\vv$ for all $ a,b\in\R$ and $\vv>0$, we obtain that for all $x,v\in\R^d,$
\begin{equation}\label{E06}
\begin{split}
\big(4\theta_0 -(\theta^*)^2\big)\bigg(\frac{1}{8 } |x|^2+\frac{ 1}{(\theta^*)^2+4\theta_0 } |v|^2\bigg)&\le \theta_0|x|^2+ |v|^2+\theta^*\<x,v\>\\
&\le
\Big(1\vee \theta_0 +\frac{\theta^*}{2}\Big)\big(|x|^2+ |v|^2\big)
\end{split}
\end{equation}
so, in view of $U(x)\ge0$,
$\mathcal W_0\ge1$ and $\lim_{|x|+|v|\to\8}\mathcal W_0(x,v)=\8$.

\begin{proposition}\label{pro4}
Assume that $({\bf  A_1})$ holds,  $\int_{\R^d} |u|^\bb\varphi(u)\,\d u<\infty$ for some $\bb\in (0,2]$, and that $U(x)\ge0$ for all $x\in\R^d$ satisfying that there exist constants \begin{equation}\label{S0}c^*>c_0^*:=\frac{(\lambda_1+\gamma)^2(\lambda_2-\lambda_1)^2}{4(2\lambda_1\lambda_2-\lambda_1^2
+4\lambda_2\gamma+3\gamma^2)}\end{equation} and $c^{**}, C_0^*\ge0$ such that for all $x\in \R^d$,
\begin{equation}\label{E05}
\<x,\nn U(x)\>\ge c^*|x|^2+c^{**}U(x)-C_0^*.
\end{equation}Then, there exist  constants $c_0,C_0>0$ so that
\begin{equation}\label{F9}
(\mathscr L \mathcal W)(x,v)\le -c_0\mathcal W(x,v)+C_0,
\end{equation}
where $\mathcal W(x,v):=\mathcal W_0(x,v)^{{\bb}/{2}}$ with  $\mathcal W_0$ being  defined in \eqref{2D}.
Hence, the Assumption $({\bf B_1})$ holds for $\mathcal W $.
\end{proposition}

\begin{proof}
For any $\rho>0$, let
$$c_{\rho,*}=\int_{\{|u|\le \rho\}}\varphi(u)\,\d u,$$
and
\begin{align*}
\Theta_{1,\rho} &=(1-c_{\rho,*})\ll_2\theta^*\big(1 +2\theta_0^{{\bb}/{2}}(2/\bb-1) \big),\quad
\Theta_{2,\rho}=2(1-c_{\rho,*})\ll_2\theta_0^{{\bb}/{2}}
\big((2/\bb-1)  \theta_0+1\big),\\
\Theta_{3,\rho}&=(1-c_{\rho,*}) \ll_2\big( 1 +2 \theta_0^{{\bb}/{2}}(2/\bb-1)\big),\qquad
\Theta_{4,\rho} =4(1-c_{\rho,*})\ll_2\theta_0^{{\bb}/{2}}(2/\bb-1).
\end{align*}
Since $\varphi(u)\,\d u$ is a probability measure, in addition to $c^*>c_0^*$, defined in \eqref{S0}, there exists  a constant  $\rho>0$ sufficiently large such that
\begin{equation}\label{F6}
\Theta_{2,\rho}<c^*\theta^*,\qquad \Theta_{3,\rho}<\ll_1+2\gamma-\theta^*,\qquad \Theta_{4,\rho}<c^{**}\theta^*
\end{equation}
and
\begin{equation}\label{F7}
c^*\theta^*-\Theta_{2,\rho}>\frac{(C^{**}_0+\Theta_{1,\rho})^2}{4(\ll_1+2\gamma-\theta^*-\Theta_{3,\rho})}=:c_{0,\rho}^*,
\end{equation}
where
$C^{**}_0:=\frac{1}{2(\lambda_2+\gamma)}  (\lambda_1+\gamma)^2(\lambda_2-\lambda_1).$
Below, we shall choose $\rho>0$ large enough such that \eqref{F6} and \eqref{F7} hold simultaneously.

Obviously,  for all $x,v\in\R^d,$
\begin{equation}\label{WW}
\nn_x\mathcal W_0(x,v)=2\nn U(x)+2\theta_0x
+\theta^* v \quad \mbox{ and } \quad \nn_v\mathcal W_0(x,v)=  2 v +\theta^* x.
\end{equation} By the chain rule, for $\mathcal W(x,v)=\mathcal W_0(x,v)^{{\bb}/{2}}$,  it follows from \eqref{E0}, \eqref{E05} and \eqref{WW} that
\begin{align*}
\big(\mathscr L\mathcal W\big)(x,v)
&\le \frac{\bb}{2}\mathcal W_0(x,v)^{^{\frac{\bb-2}{2}}}\big((2\theta_0-\theta^*\gamma)\<x,v\>-(2\gamma-\theta^*)|v|^2-c^*\theta^*|x|^2-c^{**}\theta^*U(x)
+\theta^*C_0^*\big)\\
&\quad +J(x,v) \int_{\{|u|\le \rho\}}\big(\mathcal W_0(x,u)^{{\bb}/{2}}-\mathcal W_0(x,v)^{{\bb}/{2}}\big)\varphi(u)\,\d u\\
&\quad+J(x,v) \int_{\{|u|>\rho\}}\big(\mathcal W_0(x,u)^{{\bb}/{2}}-\mathcal W_0(x,v)^{{\bb}/{2}}\big)\varphi(u)\,\d u\\
&=:\Upsilon_1(x,v)+J(x,v)\Upsilon_2(x,v)+J(x,v)\Upsilon_3(x,v).
\end{align*}
Since the function $x\mapsto x^{{\bb}/{2}}$ with $\bb\in(0,2]$ is concave on $[0,\infty)$,  the mean value theorem enables us to obtain that
\begin{equation}\label{F2}
\begin{split}
\Upsilon_2(x,v)&\le \frac{\bb}{2}\mathcal W_0(x,v)^{^{\frac{\bb-2}{2}}}\int_{\{|u|\le \rho\}}\big(\mathcal W_0(x,u)-\mathcal W_0(x,v)\big)\varphi(u)\,\d u\\
&\le\frac{\bb}{2}\rho^2-\frac{\bb c_{\rho,*}}{2}\mathcal W_0(x,v)^{^{\frac{\bb-2}{2}}}(|v|^2+\theta^*\<x,v\>),
\end{split}
\end{equation}
where in the second inequality we utilized the fact that $\varphi(\cdot)$ is
a radial function and meanwhile  used $\mathcal W_0\ge1$ and $\bb\in(0,2]$.

On the other hand, by the basic inequality $(a+b)^{\ell}\le a^{\ell}+b^{\ell}$ for all $a,b\ge0$ and $\ell\in(0,1],$   we deduce from
\eqref{E06} and $\bb\in(0,2] $ that
\begin{equation}\label{FF*}
\begin{split}
&\Upsilon_3(x,v)\\
&\le \int_{\{|u|>\rho\}}\big((1+2U(x))^{{\bb}/{2}}+(\theta_0|x|^2+ |u|^2+\theta^*\<x,u\>)^{{\bb}/{2}}-(1+2U(x))^{{\bb}/{2}}\big)\varphi(u)\,\d u\\
&\le (\theta_0|x|^2)^{{\bb}/{2}}(1-c_{\rho,*})  +\int_{\{|u|>\rho\}} |u|^\bb\varphi(u)\,\d u+(\theta^*|x|)^{{\bb}/{2}}\int_{\{|u|>\rho\}}|u|^{{\bb}/{2}} \varphi(u)\,\d u\\
&\le (\theta_0|x|^2)^{{\bb}/{2}}(1-c_{\rho,*})  +\int_{\{|u|>\rho\}} |u|^\bb\varphi(u)\,\d u+(\theta^*|x|)^{{\bb}/{2}}\bigg(\int_{\{|u|>\rho\}}|u|^{\bb} \varphi(u)\,\d u\bigg)^{{1}/{2}}\\
&\le 2(\theta_0|x|^2)^{{\bb}/{2}}(1-c_{\rho,*}) +C_{\rho,*}
\end{split}
\end{equation} with \begin{equation}\label{F*}
C_{\rho,*}:=\bigg(1+\frac{\big(\theta^*/\theta_0^{{1}/{2}}\big)^\bb}{4 (1-c_{\rho,*})}\bigg)\int_{\{|u|>\rho\}} |u|^\bb\varphi(u)\,\d u,
\end{equation}
where in    the third inequality we employed   Jensen's inequality and in the last inequality we exploited Young's inequality. Again, via
Young's inequality, for $\bb\in(0,2]$ we arrive at
\begin{align*}
|x|^\bb=\mathcal W_0(x,v)^{\frac{\bb-2}{2}}\mathcal W_0(x,v)^{\frac{2-\bb}{2}}|x|^\bb
\le  \frac{\bb}{2}\mathcal W_0(x,v)^{\frac{\bb-2}{2}}\big((2/\bb- 1)\mathcal W_0(x,v)+|x|^2\big).
\end{align*}
Plugging this back into \eqref{FF*} yields
\begin{equation}\label{F5}
\begin{split}
\Upsilon_3(x,v)
&\le \frac{\bb}{2}\mathcal W_0(x,v)^{\frac{\bb-2}{2}}\theta_0^{{\bb}/{2}}(1-c_{\rho,*})\big((4/\bb-2)\mathcal W_0(x,v)+2|x|^2\big)+C_{\rho,*}.
\end{split}
\end{equation}

Now, with the help of \eqref{F2} and \eqref{F5} and by taking the expression of $\mathcal W_0$, given in \eqref{2D},  we deduce from  $\mathcal W_0\ge1$ and $\bb\in(0,2]$ that
\begin{align*}
&\big(\mathscr L\mathcal W\big)(x,v)\\
&\le \frac{\bb}{2}\mathcal W_0(x,v)^{^{\frac{\bb-2}{2}}}\big((2\theta_0-\theta^*\gamma)\<x,v\>-(2\gamma-\theta^*)|v|^2-c^*\theta^*|x|^2-c^{**}\theta^*U(x)
\big)\\
&\quad + \frac{\bb}{2}\mathcal W_0(x,v)^{\frac{\bb-2}{2}}J(x,v)\\
&\qquad\times \Big\{- c_{\rho,*}(|v|^2+\theta^*\<x,v\>)+2\theta_0^{{\bb}/{2}}(1-c_{\rho,*})\big((2/\bb-1) (2U(x)+\theta_0|x|^2+ |v|^2+\theta^*\<x,v\>)+|x|^2\big)\Big\}\\
&\quad+C_{\rho,**}\\
&\le \frac{\bb}{2}\mathcal W_0(x,v)^{^{\frac{\bb-2}{2}}}\Big\{-\big(c^{**}\theta^*-\Theta_{4,\rho}\big)  U(x)-\big(c^*\theta^*-\Theta_{2,\rho}\big)|x|^2-\big(\ll_1+2\gamma-\theta^*- \Theta_{3,\rho}\big)|v|^2\\
&\quad\quad\qquad\qquad\qquad+\big(C^*_0(x,v)+\Xi(x,v)\big)\<x,v\>\Big\}+C_{\rho,**},
\end{align*}
where $C^*_0(x,v):=2\theta_0-\theta^*(J(x,v)+\gamma),$
$
\Xi(x,v):=(1-c_{\rho,*})J(x,v)\theta^*\big(1 +2\theta_0^{{\bb}/{2}}(2/\bb-1) \big)$ and $C_{\rho,**}:=\theta^*C_0^*+\ll_2\rho^2+\ll_2C_{\rho,*}+2\ll_2 \theta_0^{{\bb}/{2}}\big(2/\bb-1\big)
$
with $C_{\rho,*}$ being introduced in \eqref{F*}.

Note that \begin{align*}0&\le  C^*_0(x,v)+\Xi(x,v)\le 2\theta_0-\theta^*(\lambda_1+\gamma)+(1-c_{\rho,*})\lambda_2\theta^*\big(1 +2\theta_0^{{\bb}/{2}}(2/\bb-1) \big)\\
&= C^{**}_0+\Theta_{1,\rho}\\
&\le \frac{1}{2}\big(c^*\theta^*-\Theta_{2,\rho}+c_{0,\rho}^*\big)+\big(\ll_1+2\gamma-\theta^*- \Theta_{3,\rho}\big)\frac{2c_{0,\rho}^*}{c^*\theta^*-\Theta_{2,\rho}+c_{0,\rho}^*},\end{align*} where $c_{0,\rho}^*>0$ was defined by \eqref{F7}.
This  yields that
\begin{align*}
\big(\mathscr L\mathcal W\big)(x,v)
&\le\frac{\bb}{2}\mathcal W_0(x,v)^{^{\frac{\bb-2}{2}}}\bigg\{-\big(c^{**}\theta^*-\Theta_{4,\rho}\big)  U(x)-\frac{1}{2}\big(c^*\theta^*-\Theta_{2,\rho}-c_{0,\rho}^*\big)|x|^2\\
&\quad\quad\qquad\qquad\qquad-(\ll_1+2\gamma-\theta^*- \Theta_{3,\rho})\frac{c^*\theta^*-\Theta_{2,\rho}-c_{0,\rho}^*}{c^*\theta^*-\Theta_{2,\rho}+c_{0,\rho}^*}|v|^2\bigg\}+C_{\rho,**},
\end{align*}
 Consequently, \eqref{F9} follows by combining \eqref{E06}
with \eqref{F6} and  \eqref{F7}.
\end{proof}

\begin{proposition}\label{Pro:1.6}
Assume that the function $r\mapsto\varphi(r)$ is bounded and decreasing on $(0,\infty)$ so that $\varphi(r)>0$ for all $r>0$ and
$\int_{\R^d}|u|^\bb\varphi(u)\, \d u<\8$ for some $\bb\in(0,2]$, and that there exists a constant $c_0^{**}>0$ such that for all $\xi\in\R^d$,
\begin{equation}\label{0D0}
\int_{\R^d}|u|^\bb\Psi_{\xi}(u)\, \d u\le c_0^{**}|\xi| .
\end{equation}
Then,  Assumption $({\bf B}_2)$ holds for $\mathcal W(x,v)=\mathcal W_0(x,v)^{\bb/2}$,
 where $\mathcal W_0$ was defined in \eqref{2D}.
\end{proposition}

\begin{proof}
Via the inequality that $(a+b+c)^{\ell}\le a^{\ell}+b^{\ell}+c^\ell\le 3(a+b+c)^\ell$ for all $a,b,c\ge0$ and $\ell\in(0,1],$ we infer from Jensen's inequality and Young's inequality that
\begin{align*}
\int_{\R^d}\mathcal W(x,u)\varphi(u)\, \d u
&\le \big(1+2U(x)+\theta_0|x|^2\big)^{\bb/2}+\int_{\R^d}|u|^\bb\varphi(u)\, \d u+(\theta^*|x|)^{\bb/2}\int_{\R^d}|u|^{\bb/2}\varphi(u)\, \d u\\
&\le \big(1+2U(x)+\theta_0|x|^2\big)^{\bb/2} +(\theta_0|x|^2)^{\bb/2}+C^{**}\\
&\le 3\Big(1+(C^{**})^{2/\bb}\vee(2\theta_0)\Big)^{\bb/2}\big(1+2U(x)+ |x|^2\big)^{\bb/2},
\end{align*}
where
\begin{equation*}
C^{**}:=\bigg(1+\frac{1}{4} \big(\theta^*/\theta_0^{{1}/{2}}\big)^\bb\bigg)\int_{\R^d}|u|^{\bb} \varphi(u)\,\d u<\8
\end{equation*}
by taking $\int_{\R^d}|u|^{\bb} \varphi(u)\,\d u<\8$ into consideration.
Thus, we arrive at
\begin{align*}
\int_{\R^d}\mathcal W(x,u)\varphi(u)\, \d u
&\le 3\Big(1+(C^{**})^{2/\bb}\vee(2\theta_0)\Big)^{\bb/2}\inf_{v\in\R^d}\big(1+2U(x)+ |x|^2+|v|^2\big)^{\bb/2}.
\end{align*}
Therefore, the first inequality in \eqref{e11} holds true for $\mathcal W$.

Next, by taking advantage of the inequality that $(a+b)^{\ell}\le a^{\ell}+b^{\ell}$ for all $a,b\ge0$ and $\ell\in(0,1],$ and by invoking H\"older's inequality, we find that
\begin{align*}
\int_{\R^d}\mathcal W(x,u) \Psi_{\xi}(u)\, \d u
&\le \big(1+2U(x)+\theta_0|x|^2\big)^{\bb/2}\int_{\R^d}\Psi_{\xi}(u)\, \d u+\int_{\R^d}|u|^\bb\Psi_{\xi}(u)\, \d u\\
&\quad+(\theta^*|x|)^{\bb/2}\bigg(\int_{\R^d}|u|^\bb\Psi_{\xi}(u)\, \d u\bigg)^{{1}/{2}}\bigg(\int_{\R^d} \Psi_{\xi}(u)\, \d u\bigg)^{{1}/{2}}.
\end{align*}
Then, applying Proposition \ref{pro} and taking \eqref{0D0} into account yield that for some constant $C_{**}>0$,
\begin{equation*}
\int_{\R^d}\mathcal W(x,u) \Psi_{\xi}(u)\, \d u
\le C_{**}\big(1+2U(x)+ |x|^2\big)^{\bb/2}|\xi|.
\end{equation*}
Consequently, the second inequality in \eqref{e11} is proved thanks to \eqref{E06} again.

By summing up the previous analysis, we make a conclusion that Assumption $({\bf B}_2)$ is  provable for $\mathcal W(x,v)=\mathcal W_0(x,v)^{\bb/2}$.
\end{proof}

Examples for the probability density function $\varphi$ that satisfies all the assumptions in Propositions \ref{pro4} and \ref{Pro:1.6} are
$\varphi(x)=\varphi_1(x):=c_{d,\beta_1} (1+|x|)^{-d-\beta_1}$ with $\beta_1>0$ and $c_{d,\beta_1}>0$ or $\varphi(x)=\varphi_2(x):= c_{d,\beta_2} \exp (-|x|^{\beta_2})$ with $\beta_2>0$ and $c_{d,\beta_2}>0$.

\ \

Finally, we intend to  validate  the condition \eqref{S6}.

\begin{proposition}
Assume $U(x)=\theta|x|^2$ for  all $\theta>0$ and $x\in\R^d.$ Then, the inequality \eqref{S6} is solvable in case of $\gamma\ge2\ss{2\theta}$.
\end{proposition}

 \begin{proof}
 Due to $ \gamma\ge 2\ss{2\theta}$, for any
$$\alpha\in\Big(0, \Big(\gamma-\ss{\gamma^2-8\theta}\Big)/2\Big)\bigcup\Big( \Big(\gamma+\ss{\gamma^2-8\theta}\Big)/2,+\8\Big),$$
we have $\alpha(\gamma-\alpha)\le2\theta$. On the other hand, for all
\begin{equation*}
\frac{1}{2}\Big(\gamma-\ss{\gamma^2-32\theta/5}\Big)<\alpha<\frac{1}{2}\Big(\gamma+\ss{\gamma^2-32\theta/5}\Big)
\end{equation*}
provided  $\gamma\ge\ss{32\theta/5}$, it holds that
$
\alpha^2-\gamma\alpha+ 8\theta/5\le0.
$
Hence, for  $\gamma\ge 2\ss{2\theta}$, we find that there exists  an $\alpha>0$ satisfying simultaneously
\begin{equation}\label{S18}
\alpha(\gamma-\alpha)\le2\theta \end{equation}
and
\begin{equation}\label{S19}
5\alpha^2-5\gamma\alpha+8\theta\le0 .
\end{equation}

Thanks to  $U(x)=\theta|x|^2$, $x\in\R^d,$ we obviously obtain that   for all $x,x'\in\R^d,$
\begin{equation*}
\nn U(x)-\nn U(x')=2\theta(x-x').
\end{equation*}
Thus, for $\bb:=\alpha(\gamma-\alpha)\le 2\theta$ due to \eqref{S18}, we deduce that  for all $x,x'\in\R^d,$
\begin{equation*}
|\bb(x-x')+\nn U(x' )-\nn U(x )|=(2\theta-\beta)|x-x'|.
\end{equation*}
Therefore, Assumption (${\bf H_0}$) is satisfied for $K_{\beta,U}=2\theta-\beta$.

Next, it is clear that $\beta\le 2\theta\le \gamma^2/4$.
On the other hand, with \eqref{S19} at hand, we find that for $\alpha>0$ solving \eqref{S18} and \eqref{S19},
\begin{equation*}
\beta=\gamma\alpha-\alpha^2\ge 4(2\theta-\alpha(\gamma-\alpha))=4K_{\alpha(\gamma-\alpha),U}=4K_{\beta,U}
\end{equation*}
by recalling $K_{\beta,U}=2\theta-\beta$ with $\beta=\alpha(\gamma-\alpha)$. Consequently, we can reach  a conclusion that the inequality \eqref{S6} is solvable.
\end{proof}

With all the propositions above at hand, we can easily verify Example \ref{Example}, and so the detail is omitted here to save space.

\ \

\noindent{\bf Acknowledgement}\,\, The research of Jianhai Bao is
supported by the National Natural Science Foundation of China (Nos.\ 12071340 and
11831014).
The research of Jian Wang is supported
by the National Natural Science Foundation of China (Nos.\ 11831014
and 12071076), and the Education and Research Support Program for Fujian Provincial Agencies.

\end{document}